\pdfoutput=1
\documentclass[11pt]{article}

\usepackage[margin=1.1in]{geometry}
\usepackage{amsmath,amssymb,amsthm}
\usepackage{mathtools}
\usepackage{booktabs}
\usepackage{microtype}
\usepackage{xcolor}
\usepackage[colorlinks=true,linkcolor=blue!55!black,citecolor=blue!55!black,urlcolor=blue!55!black]{hyperref}
\usepackage[capitalize,noabbrev]{cleveref}

\theoremstyle{plain}
\newtheorem{theorem}{Theorem}
\newtheorem{proposition}{Proposition}
\newtheorem{corollary}{Corollary}
\theoremstyle{definition}
\newtheorem*{definitionx}{Definition}
\theoremstyle{remark}
\newtheorem{remark}{Remark}

\newcommand{\Z}{\mathbb{Z}}
\newcommand{\C}{\mathbb{C}}
\newcommand{\F}{\mathbb{F}}
\newcommand{\Nmin}{N_{\min}}
\DeclareMathOperator{\per}{per}

\title{Ryser, Glynn, and the discrete Fourier transform:\\
orthogonal schemes for the permanent}
\author{Jos\'e A.\ R.\ Fonollosa\thanks{Universitat Polit\`ecnica de
Catalunya, Barcelona, Spain. \texttt{jose.fonollosa@upc.edu}. ORCID:
\href{https://orcid.org/0000-0001-9513-7939}{0000-0001-9513-7939}.}}
\date{July 2026}

\begin{document}
\maketitle

\begin{abstract}
The permanent of an $n \times n$ matrix is the coefficient of the fully mixed
monomial $x_0 \cdots x_{n-1}$ in the product of its row forms
$\prod_i\bigl(\sum_j b_{ij} x_j\bigr)$, and the classical exact algorithms of
Ryser and Glynn compute it by summing $2^{n-1}$ evaluations of that product over the
Boolean cube. We recast this as a single principle --- an \emph{evaluation
scheme} together with an orthogonality criterion that decides, in one line,
whether its weighted sum equals the permanent --- and show that Ryser's formula,
Glynn's formula, and a discrete Fourier transform are three instances of it: the
$0$--$1$ and $\pm 1$ Hadamard systems, and a single cyclic transform of length
$N$. The DFT scheme evaluates $\per B$ as one coefficient of a univariate
polynomial modulo $x^N - 1$; it is exact exactly when the exponent set is
\emph{valid} mod $N$, runs in the same $\Theta(2^n n)$ operations as the
classical formulas, and over a finite field is a number-theoretic transform that
returns the exact integer permanent by the Chinese remainder theorem. Where the
Hadamard schemes are orthogonal for free, the cyclic scheme
trades this for an arithmetic condition on the exponents; the smallest length at
which it can be met is $N = 2^n - 2^{\lfloor \log_2 n \rfloor}$, established in
the companion paper \cite{fonollosa2026minmod}. All three are character sums over
a finite abelian group, and the criterion holds precisely for such groups.


\end{abstract}

\section{Introduction}

The permanent of an $n \times n$ matrix $B = (b_{ij})$,
\[
\per B \;=\; \sum_{\sigma \in S_n} \prod_{i=1}^{n} b_{i\sigma(i)},
\]
is \#P-complete to compute exactly, even for $0$--$1$ matrices
\cite{valiant1979}. The fastest known exact algorithms are Ryser's
inclusion--exclusion formula \cite{ryser1963,nijenhuis1978} and Glynn's
polarization identity \cite{glynn2010}, each a weighted sum of $2^{n-1}$ terms
and each running in $\Theta(2^n n)$ arithmetic operations.

Both formulas are instances of one idea. Writing $\per B$ as the coefficient of
the fully mixed monomial $x_0 \cdots x_{n-1}$ in the row-product polynomial
$\prod_i\bigl(\sum_j b_{ij} x_j\bigr)$, one extracts that coefficient by
evaluating the product at a structured family of points and combining the values
with fixed weights: Ryser evaluates over $\{0,1\}^n$ with inclusion--exclusion
signs, Glynn over $\{\pm 1\}^n$ with a parity weight, halved to $2^{n-1}$ points
by antipodal symmetry. This paper makes the common
structure explicit --- an \emph{evaluation scheme} and a one-line
\emph{orthogonality criterion} that decides when its weighted sum is the
permanent (\cref{prop:ortho}) --- and adds a third instance that is neither
Boolean nor a Hadamard transform.

The new instance collapses the $n$ variables onto powers of a single one,
$x_j \mapsto x^{a_j}$, and works modulo $x^N - 1$: then $\per B$ is a single
coefficient of a univariate polynomial of degree ${<}N$, read off by one
discrete Fourier transform of length $N$ (\cref{prop:dft}). This evaluation is
exact precisely when the exponent set $A = \{a_j\}$ is \emph{valid} mod $N$ ---
no size-$n$ multiset other than the all-ones one sums to $p = \sum_j a_j$ modulo
$N$. Over a finite field it becomes a number-theoretic transform and, across
several primes with the Chinese remainder theorem, returns the exact integer
permanent with no rounding.

The two kinds of scheme sit in complementary regimes. The Hadamard schemes are
orthogonal \emph{unconditionally} --- they compute the permanent of every matrix
--- each at $2^{n-1}$ evaluation points after its natural halving (Glynn's
$\delta_0 = 1$, the Nijenhuis--Wilf reduction of Ryser). The cyclic DFT instead
trades unconditional orthogonality for an arithmetic condition on the exponents:
it is exact only for \emph{valid} $(A, N)$, and the least valid length is
$\Nmin(n) = 2^n - 2^{\lfloor \log_2 n \rfloor} \approx 2^n$, determined in the
companion paper \cite{fonollosa2026minmod}, whose combinatorial \emph{unique
multiset-sum problem} is exactly the validity condition used here. All three reach
the same $\Theta(2^n n)$ arithmetic.

\paragraph{Related work.}
That Ryser's and Glynn's formulas share one pattern is, in essence, Glynn's own
observation: his formula is derived from a polarization identity for symmetric
tensors in \cite{glynn2010}, and \cite{glynn2013} develops from it a whole
family of permanent formulae, parametrized via the Veronese variety, containing
both. Character sums beyond $\pm 1$ have also appeared: Nilsson
\cite{nilsson2017} rederives Glynn's formula from the $\Z_2$ gauge invariance of
a Grassmann-integral representation and records its $p$-th-roots-of-unity
analogue over $(\Z_p)^n$, and the boson-sampling literature contains
roots-of-unity generalizations of Glynn's formula to matrices with repeated rows
and columns (see \cite{chabaud2022} and the references therein). All of these
evaluate over a full product group with the standard embedding --- $p^n$
points --- so no condition on the exponents ever arises. What is new here, to
our knowledge, is the criterion as an exact characterization
(\cref{prop:ortho,prop:abelian}) and its \emph{cyclic} instance --- a single
variable, $x_j \mapsto x^{a_j}$ modulo $x^N - 1$ --- where orthogonality becomes
an arithmetic condition on the exponent set and the least admissible transform
length becomes a well-posed combinatorial question, answered by
\cref{thm:minmod}. In the language of algebraic complexity, every evaluation
scheme is a depth-three ($\Sigma\Pi\Sigma$) arithmetic formula for the
permanent of a special shape --- each product gate multiplies one linear form
per row (a \emph{set-multilinear} formula), here with the same coefficient
vector $v_r$ in every row --- so a scheme with $|R|$ points writes the
permanent tensor as a sum of $|R|$ rank-one terms. In this class the
$2^{\Theta(n)}$ point count is forced: the partial derivatives of $\per$ span
a space of dimension $\binom{2n}{n}$ (the subpermanents, linearly independent
over every field), while those of a product of $n$ row-linear forms span at
most $2^n$, so $|R| \ge \binom{2n}{n}/2^n \sim 2^n\!/\sqrt{\pi n}$ by the
partial-derivative method of Nisan and Wigderson \cite{nisan1997}. The
$\approx 2^n$ count shared by the three schemes is therefore optimal in this
class up to a polynomial factor; what \cref{thm:minmod} and
\cref{rem:noncyclic} add is the exact minimum within the character-scheme
subclass. Read in the other direction through \cref{prop:abelian}, the
complexity bound becomes combinatorial: any finite abelian group admitting a
size-$n$ family with unique multiset sums has order at least
$\binom{2n}{n}/2^n$, within a factor $O(\sqrt n\,)$ of Glynn's $2^{\,n-1}$
(\cref{cor:nwbound}).

\paragraph{Organization.}
\Cref{sec:framework} sets up the general evaluation scheme, proves the
orthogonality criterion, and recovers Ryser, Glynn, and the cyclic DFT as its
three instances. \Cref{sec:identity} then develops the cyclic instance in detail
--- the explicit transform, its number-theoretic variant, a numerical validity
test, and cost. All computational statements here are asymptotic; a practical GPU
implementation, with timings, is beyond the scope of this paper.

\section{An orthogonality criterion}\label{sec:framework}

All three algorithms rest on a single principle: $\per B$ is the fully mixed
coefficient of the row-product polynomial, and \emph{any} family of evaluations
whose weighted sum isolates that coefficient computes it. We state the principle
in general, read Ryser's and Glynn's formulas off it, and defer the cyclic
instance new to this paper to \cref{sec:identity}.

\paragraph{Coefficient extraction.}
For $B \in \C^{n \times n}$ the row-product polynomial
\[
P_B(x_0, \dots, x_{n-1}) \;=\; \prod_{i=1}^n\Bigl(\sum_{j=0}^{n-1} b_{ij}\, x_j\Bigr)
\]
is homogeneous of degree $n$, and $\per B$ is its coefficient of the fully mixed
monomial $x_0 x_1 \cdots x_{n-1}$. Expanding, each map
$f : \{1,\dots,n\} \to \{0,\dots,n-1\}$ contributes $\prod_i b_{i f(i)}$ times
$\prod_j x_j^{k_j}$, where $k_j = |f^{-1}(j)|$ and $\sum_j k_j = n$; the mixed
monomial is contributed exactly by the bijections, i.e.\ by the multiplicity
vector $k = \mathbf 1 := (1,\dots,1)$.

\begin{definitionx}
An \emph{evaluation scheme} of order $n$ is a finite family of vectors
$v_r \in \C^n$ ($r \in R$) with weights $w_r \in \C$ and a constant $c \in \C$.
Its \emph{transform sum} is
\[
T[B] \;=\; c \sum_{r \in R} w_r \prod_{i=1}^n \Bigl(\sum_{j=0}^{n-1} b_{ij}\, v_{r,j}\Bigr),
\qquad B \in \C^{n \times n}.
\]
\end{definitionx}

\begin{proposition}[orthogonality criterion]\label{prop:ortho}
For each multiplicity vector $k \in \Z_{\ge 0}^{\,n}$ with $|k| = n$ put
\[
\Phi(k) \;=\; c \sum_{r \in R} w_r \prod_{j=0}^{n-1} v_{r,j}^{\,k_j}
\qquad (\text{with } t^0 := 1).
\]
Then $T[B] = \per B$ for every $B \in \C^{n\times n}$ if and only if
$\Phi(k) = [\,k = \mathbf 1\,]$ for all such $k$.
\end{proposition}

\begin{proof}
Expanding the product and collecting terms by $f$,
\[
T[B] \;=\; \sum_f \Bigl(\prod_{i=1}^n b_{i f(i)}\Bigr)\, c\sum_{r} w_r \prod_{i} v_{r,f(i)}
\;=\; \sum_f \Bigl(\prod_{i=1}^n b_{i f(i)}\Bigr)\, \Phi\bigl(k(f)\bigr),
\]
since $\prod_i v_{r,f(i)} = \prod_j v_{r,j}^{\,k_j(f)}$ depends on $f$ only through
its multiplicity vector $k(f)$. Regarding the $b_{ij}$ as indeterminates, the
monomials $\prod_i b_{i f(i)}$ are pairwise distinct --- the row-$i$ factor
$b_{i f(i)}$ recovers $f(i)$ --- hence linearly independent, so two such sums
agree for all $B$ if and only if they match coefficientwise. As
$\per B = \sum_f [\,k(f) = \mathbf 1\,]\prod_i b_{i f(i)}$, and every admissible
$k$ equals $k(f)$ for some $f$, the identity holds for all $B$ iff
$\Phi(k) = [\,k = \mathbf 1\,]$ throughout.
\end{proof}

Two of the three schemes below meet this criterion for every matrix; the third,
cyclic one meets it only under a combinatorial condition on its exponents, which
we name now.

\begin{definitionx}
A set $A = \{a_0 < \dots < a_{n-1}\} \subseteq \Z_N$ is \emph{valid mod $N$} if
the all-ones multiset is the only size-$n$ multiset with elements from $A$ whose
sum is $p := \sum_j a_j \pmod N$.
\end{definitionx}

\begin{proposition}[three orthogonal schemes]\label{prop:three}
Each scheme below satisfies $\Phi(k) = [\,k = \mathbf 1\,]$, and hence computes
$\per B$ by \cref{prop:ortho}.
\begin{itemize}
\item[\textup{(R)}] \emph{Ryser --- Hadamard $0$--$1$.} $R = \{0,1\}^n$, with $v_S
= \chi_S$ the indicator of $S \subseteq \{0,\dots,n-1\}$, $w_S = (-1)^{\,n-|S|}$,
$c = 1$ \textup{\cite{ryser1963,nijenhuis1978}}.
\item[\textup{(G)}] \emph{Glynn --- Hadamard $\pm 1$.}
$R = \{\delta \in \{\pm 1\}^n : \delta_0 = 1\}$, $v_\delta = \delta$,
$w_\delta = \prod_{j} \delta_j$, $c = 2^{-(n-1)}$ \textup{\cite{glynn2010}}.
\item[\textup{(F)}] \emph{This paper --- cyclic DFT.} $R = \{0,\dots,N-1\}$,
$v_{r,j} = \omega^{a_j r}$ with $\omega = e^{2\pi i/N}$, $w_r = \omega^{-pr}$,
$c = 1/N$; orthogonal \emph{if and only if $A$ is valid mod $N$}.
\end{itemize}
\end{proposition}

\begin{proof}
\textup{(R)} With $t^0 = 1$ one has $\prod_j \chi_S(j)^{k_j} =
[\,\operatorname{supp} k \subseteq S\,]$, so, using
$\sum_{T \subseteq U}(-1)^{|T|} = [\,U = \varnothing\,]$,
\[
\Phi(k) = (-1)^n\!\!\sum_{S \,\supseteq\, \operatorname{supp} k}\!\!(-1)^{|S|}
= (-1)^{\,n + |\operatorname{supp} k|}\,\bigl[\,\operatorname{supp} k = \{0,\dots,n-1\}\,\bigr]
= [\,k = \mathbf 1\,],
\]
the last equality because $\operatorname{supp} k = \{0,\dots,n-1\}$ with $|k| = n$
forces $k = \mathbf 1$.

\textup{(G)} Fixing $\delta_0 = 1$ discards the redundant antipodal pairs
$\delta \leftrightarrow -\delta$ --- the summand $\prod_j \delta_j^{\,k_j+1}$ is
unchanged under $\delta \mapsto -\delta$, since $\sum_j (k_j+1) = 2n$ is even ---
leaving
$\displaystyle \Phi(k) = 2^{-(n-1)}\!\!\sum_{\delta_0 = 1}\prod_j \delta_j^{\,k_j+1}
= \prod_{j\ge1}\Bigl(\tfrac12\!\!\sum_{\delta_j = \pm1}\!\!\delta_j^{\,k_j+1}\Bigr)
= \prod_{j\ge1} [\,k_j \text{ odd}\,]$,
which is $1$ iff $k_1, \dots, k_{n-1}$ are all odd. Each $k_j - 1$ ($j \ge 1$) is
then a nonnegative even integer with $\sum_{j\ge1}(k_j-1) = 1 - k_0 \le 1$, so
every $k_j = 1$ for $j \ge 1$ and $k_0 = 1$; that is, $k = \mathbf 1$.

\textup{(F)}
$\displaystyle \Phi(k) = \tfrac1N \sum_{r=0}^{N-1} \omega^{\,r(\sum_j k_j a_j - p)}
= \bigl[\,\textstyle\sum_j k_j a_j \equiv p \!\!\pmod N\,\bigr]$
by character orthogonality; this equals $[\,k = \mathbf 1\,]$ for all $k$ with
$|k| = n$ exactly when no multiset other than the all-ones one reaches $p$
modulo $N$ --- the definition of validity.
\end{proof}

\Cref{tab:schemes} collects the three schemes side by side, with the mechanism
by which each meets the criterion.

\begin{table}[h]
\centering
\small
\begin{tabular}{@{}lcccc@{}}
\toprule
scheme & values $v_{r,j}$ & points $|R|$ & weight $w_r$ & orthogonal \\
\midrule
Ryser \textup{(R)}      & $\{0,1\}$              & $2^n$              & $(-1)^{\,n-|S|}$ & always \\
Glynn \textup{(G)}      & $\{-1,+1\}$            & $2^{n-1}$          & $\prod_j \delta_j$ & always \\
this paper \textup{(F)} & $\{\omega^{a_j r}\}$   & $\Nmin = 2^n - 2^m$ & $\omega^{-pr}$   & iff valid mod $N$ \\
\bottomrule
\end{tabular}
\caption{The three orthogonal schemes of \cref{prop:three} (here
$m = \lfloor \log_2 n \rfloor$): the Boolean Hadamard faces \textup{(R)},
\textup{(G)}, orthogonal for every matrix --- Ryser's inclusion--exclusion over
all $2^n$ subsets and Glynn's reduced $2^{n-1}$-point form --- and the cyclic
scheme \textup{(F)}, orthogonal exactly when $A$ is valid mod
$N$.}\label{tab:schemes}
\end{table}

\paragraph{Two regimes of orthogonality.}
The schemes divide by \emph{how} the criterion is met. \textup{(R)} and
\textup{(G)} are the $0$--$1$ and $\pm 1$ faces of the same order-$2^n$ Hadamard
system --- the characters of $(\Z_2)^n$ --- and $\Phi(k) = [\,k = \mathbf 1\,]$
holds \emph{identically}, for every matrix and every $n$. On the $\pm 1$ face the
antipodal symmetry $\delta \leftrightarrow -\delta$ gives every term twice, so one
keeps a single representative of each pair $\{\delta, -\delta\}$ and evaluates only
$2^{n-1}$ points: Glynn does this by fixing $\delta_0 = 1$, the form stated in
\textup{(G)} \textup{\cite{glynn2010}}, and Nijenhuis and Wilf give a variant of
Ryser's method over the same $2^{n-1}$ sign patterns --- one from each
complementary pair --- \textup{\cite{nijenhuis1978}}. Scheme \textup{(F)} instead
replaces the $n$-fold order-$2$ transform by a single order-$N$ cyclic one, for
which orthogonality is no longer automatic but is exactly validity mod $N$. How
small $N$ can be made is itself a combinatorial question --- the \emph{unique
multiset-sum problem} --- answered for the natural super-increasing set in the
companion paper.

\begin{theorem}[least valid modulus; \cite{fonollosa2026minmod}]\label{thm:minmod}
For the super-increasing set $A = \{2^k - 1 : 0 \le k \le n-1\}$ and every
$n \ge 2$, the least $N$ for which $A$ is valid mod $N$ is
\[
\Nmin(n) \;=\; 2^{\,n} - 2^{\lfloor \log_2 n \rfloor}.
\]
Conjecturally no size-$n$ residue set is valid below this bound.
\end{theorem}

The bound is easy to motivate, even though its proof (in
\cite{fonollosa2026minmod}) is not. Translating the set to the geometric ladder
$a_j = 2^j$ (\cref{rem:translation}), a size-$n$ multiset with multiplicities
$k = (k_0, \dots, k_{n-1})$, $\sum_j k_j = n$, has value $\sum_j k_j 2^j$, and the
all-ones multiset has value $2^n - 1$; a rival is a \emph{different} $k$ of the
same size with $\sum_j k_j 2^j \equiv 2^n - 1 \pmod N$, i.e.\
$\sum_j k_j 2^j = (2^n - 1) + rN$ for some integer $r \ne 0$. A minimal-digit-sum
estimate shows that every nonzero multiple $rN$ forces any binary representation
of its target to spend strictly more than $n$ coins once $N \ge 2^n -
2^{\lfloor \log_2 n \rfloor}$, so no rival of size exactly $n$ survives; just
below that modulus an explicit rival appears. The correction $2^{\lfloor \log_2 n
\rfloor}$ is precisely the slack at the tight boundary $2^{\lfloor \log_2 n
\rfloor} \le n \le 2^{\lfloor \log_2 n \rfloor + 1} - 1$.

This $\Nmin(n)$ is below $2^n$, and for a \emph{real} matrix the conjugate
symmetry $\omega \leftrightarrow \omega^{-1}$ (\cref{rem:translation}) halves it
to $\approx 2^{n-1}$, the same order as Glynn's reduced form; over $\C$ or a
finite field no such halving applies. The three schemes are, in the end, on the
same footing computationally: the Hadamard schemes use on the order of $2^{n-1}$
evaluation points and the cyclic transform $\approx 2^n$, but all reach the same
$\Theta(2^n n)$ arithmetic operations (below), each attaining its
$\Theta(n)$-per-point cost only through an ordered, incremental traversal of its
evaluation points. What separates \textup{(F)} is not its cost
but its nature: it is the one \emph{cyclic} instance, orthogonal by an arithmetic
condition rather than for free.

\paragraph{The weight is the target's dual character.}
In every scheme $w_r$ is the value at $r$ of the character dual to the mixed
monomial: $\prod_j \delta_j$ is the $(\Z_2)^n$-character of the full coordinate
set \textup{(G)}; $(-1)^{\,n-|S|}$ is the M\"obius weight selecting full support
\textup{(R)}; and $\omega^{-pr}$ is the $\Z_N$-character dual to the target
exponent $p = \sum_j a_j$ \textup{(F)}. Reading off the permanent is, in each
case, the projection of the evaluated row-products onto that one character ---
Glynn's closing multiplication by a $\pm 1$ vector before summing is exactly
this dual-character weighting. Seen this way, the unique multiset-sum problem of
\cref{thm:minmod} is the question of the smallest \emph{cyclic} transform that
orthogonalizes the permanent, and \cref{prop:ortho,prop:three} set Ryser, Glynn,
and this transform on a common footing.

\paragraph{Arbitrary abelian groups.}
Nothing in \cref{prop:three} is special to $\Z_N$ or to $(\Z_2)^n$: both are
instances of one construction over an arbitrary finite abelian group. Let $G$ be
a finite abelian group of order $N$, written additively, with character group
$\hat G$, and fix $g_0, \dots, g_{n-1} \in G$ with sum $t = \sum_j g_j$. The
\emph{character scheme} over $(G, g)$ takes $R = \hat G$, $v_{\chi,j} = \chi(g_j)$,
weight $w_\chi = \overline{\chi(t)}$, and $c = 1/N$.

\begin{proposition}[abelian-group schemes]\label{prop:abelian}
The character scheme over $(G, g)$ satisfies $\Phi(k) = [\,k = \mathbf 1\,]$ ---
and hence computes $\per B$ --- if and only if $g_0, \dots, g_{n-1}$ have
\emph{unique multiset sums} in $G$: the only $k \in \Z_{\ge 0}^{\,n}$ with
$|k| = n$ and $\sum_j k_j g_j = t$ is $k = \mathbf 1$.
\end{proposition}

\begin{proof}
Since $\prod_j \chi(g_j)^{k_j} = \chi\bigl(\sum_j k_j g_j\bigr)$ and
$\frac1N \sum_{\chi \in \hat G} \chi(h) = [\,h = 0\,]$ (orthogonality of the
characters of $G$),
\[
\Phi(k) = \frac1N \sum_{\chi \in \hat G} \overline{\chi(t)}\,
\chi\Bigl(\textstyle\sum_j k_j g_j\Bigr)
= \Bigl[\,\textstyle\sum_j k_j g_j = t \ \text{in } G\,\Bigr],
\]
which equals $[\,k = \mathbf 1\,]$ for all admissible $k$ exactly under the stated
uniqueness.
\end{proof}

The Boolean and cyclic schemes are the two extremes of this construction. The
unreduced $\pm 1$ Hadamard --- $G = (\Z_2)^n$ with $g_j = e_j$ the standard
generators --- has uniqueness \emph{automatic}: $\sum_j k_j e_j = \mathbf 1$
forces every $k_j$ odd, hence $k = \mathbf 1$, at the price $N = 2^n$. Glynn's
reduced form \textup{(G)}, fixing $\delta_0 = 1$, is precisely the character
scheme over the quotient $(\Z_2)^n / \langle e_0 \rangle \cong (\Z_2)^{n-1}$ with
embedding $g = (0, e_1, \dots, e_{n-1})$ --- its characters are the $\delta_0 = 1$
evaluations --- and this embedding still has unique multiset sums
(\cref{rem:noncyclic}), so the permanent is already computed from a group of order
$2^{n-1}$. The intermediate products $G = (\Z_p)^n$ with the standard embedding
have automatic uniqueness by the same argument taken mod $p$; the resulting
$p^n$-point scheme is Nilsson's roots-of-unity formula
\textup{\cite[Eq.~(21)]{nilsson2017}}. Taking $G = \Z_N$ with $g_j = a_j$
recovers the cyclic scheme
\textup{(F)}, where uniqueness is the validity condition and $N$ can fall below
$2^n$ (Ryser is the inclusion--exclusion companion of the $(\Z_2)^n$ case). Every
finite abelian group in between yields an exact permanent evaluator whenever its
embedding has unique multiset sums, and \cref{thm:minmod} --- the least valid
\emph{cyclic} $N$ --- is one case of a broader question: the least $|G|$, over
\emph{all} finite abelian groups, admitting a size-$n$ embedding with unique
multiset sums. Commutativity is used only to collapse $\prod_j \chi(g_j)^{k_j}$ to
a single character, so abelian groups are the natural limit of the construction.

\begin{remark}\label{rem:noncyclic}
The least order is \emph{not} attained by a cyclic group. Over $G = (\Z_2)^{n-1}$
take $g = (0, e_1, \dots, e_{n-1})$, with target $t = \sum_j g_j = \mathbf 1$; any
$k$ with $|k| = n$ and $\sum_j k_j g_j = t$ has $\sum_{j\ge1} k_j e_j = \mathbf 1$,
forcing $k_1, \dots, k_{n-1}$ odd, whence $k = \mathbf 1$ exactly as in
\cref{prop:three}\,\textup{(G)}. So this order-$2^{n-1}$ group carries a valid
size-$n$ embedding --- it is Glynn's $2^{n-1}$-term form read as a character scheme
--- giving the upper bound $\min_G |G| \le 2^{n-1}$. Since
$2^{n-1} < 2^n - 2^{\lfloor \log_2 n \rfloor}$ for every $n \ge 3$, this
non-cyclic group is strictly smaller than the least valid cyclic modulus of
\cref{thm:minmod}: no cyclic evaluator is optimal. An exhaustive search over all
finite abelian groups confirms $2^{n-1}$ is exactly the least order for
$2 \le n \le 6$, and \cite{fonollosa2026minmod} proves it least among elementary
abelian $2$-groups for all $n$; whether $2^{n-1}$ stays optimal over \emph{all}
abelian groups for every $n$ is open --- though by \cref{cor:nwbound} below it
is optimal to within a factor $O(\sqrt n\,)$.
\end{remark}

The size of \emph{any} evaluation scheme --- character-based or not --- is
bounded below by the partial-derivative method of arithmetic-circuit
complexity \cite{nisan1997}, and through \cref{prop:abelian} the bound becomes
a purely combinatorial statement about abelian groups.

\begin{corollary}[lower bound for schemes and for abelian groups]\label{cor:nwbound}
Every evaluation scheme of order $n$ computing the permanent has
$|R| \ge \binom{2n}{n}/2^{\,n} \sim 2^{\,n}\!/\sqrt{\pi n}$. Consequently
every finite abelian group $G$ admitting a size-$n$ embedding with unique
multiset sums satisfies $|G| \ge \binom{2n}{n}/2^{\,n}$, so by
\cref{rem:noncyclic}
\[
\binom{2n}{n}\Big/2^{\,n} \;\le\; \min_G |G| \;\le\; 2^{\,n-1},
\]
and $(\Z_2)^{n-1}$ is optimal to within a factor $\sqrt{\pi n}/2$.
\end{corollary}

\begin{proof}
For a polynomial $f$ in the entries $b_{ij}$ let $\partial(f)$ denote the
linear span of $f$ together with all its partial derivatives, of every order
\cite{nisan1997}. Differentiating $\per B$ with respect to a set of entries
gives the permanent of the complementary minor when the set has no repeated
row or column, and $0$ otherwise; subpermanents of distinct minors have
disjoint monomial supports --- the row and column sets of any monomial recover
the minor --- so they are linearly independent and
$\dim \partial(\per) = \sum_{k=0}^{n} \binom{n}{k}^{2} = \binom{2n}{n}$.
For a single product gate $P_r = \prod_{i=1}^n \bigl(\sum_j b_{ij}
v_{r,j}\bigr)$, differentiating kills any factor hit twice and replaces a
factor hit once by a constant, so every partial derivative of $P_r$ is a
scalar multiple of $\prod_{i \notin T} \bigl(\sum_j b_{ij} v_{r,j}\bigr)$ for
$T$ the set of rows differentiated: $\dim \partial(P_r) \le 2^{\,n}$. As
$\partial(\cdot)$ is subadditive over linear combinations,
$\per = c\sum_{r \in R} w_r P_r$ forces
$\binom{2n}{n} \le |R|\,2^{\,n}$; the asymptotics is Stirling. For the second
statement, a size-$n$ embedding with unique multiset sums in $G$ makes the
character scheme --- $|R| = |G|$ points --- compute the permanent
(\cref{prop:abelian}). Both counting arguments use only disjointness of
monomial supports, so the bound holds over every field.
\end{proof}

\begin{remark}
\cref{prop:ortho} uses only that distinct monomials in the entries $b_{ij}$ are
linearly independent, so it holds over any commutative ring: \textup{(R)} and
\textup{(G)} compute the permanent over $\Z$ directly, and \textup{(F)} does so
over $\F_q$ for a prime $q \equiv 1 \pmod N$ (\cref{rem:ntt}) --- the three
schemes share one criterion across $\C$, $\Z$, and finite fields.
\end{remark}

\section{The cyclic transform}\label{sec:identity}

Scheme \textup{(F)} of \cref{prop:three} --- the one new to this paper, and the
only one whose orthogonality is conditional --- rewards a closer look. Collapsing
the $n$ variables of $P_B$ onto powers of a single one, $x_j \mapsto x^{a_j}$ for
a set $A = \{a_0, \dots, a_{n-1}\}$ of distinct nonnegative integers, turns the
row-product polynomial into a univariate one, in which $\per B$ survives as the
coefficient of $x^{\,p}$, $p = \sum_j a_j$ --- \emph{provided} no other size-$n$
multiset of the $a_j$'s reaches the same total $p$. Reducing modulo $x^N - 1$
folds all exponents into $\Z_N$; the fold is harmless exactly when $A$ is valid
mod $N$, and the coefficient of $x^{\,p}$ is then recovered by a size-$N$ discrete
Fourier transform.

Throughout this section $A = \{a_0, \dots, a_{n-1}\} \subseteq \{0, 1, \dots,
N-1\}$ is a set of $n$ distinct residues, $p \equiv \sum_j a_j \pmod N$, and
$B = (b_{ij}) \in \C^{n \times n}$.

Expanding the product of the row polynomials
$\prod_{i}\bigl(\sum_{j} b_{ij} x^{a_j}\bigr)$ produces one term per function
$f : \{1,\dots,n\} \to \{0,\dots,n-1\}$ (row $i$ picks column $f(i)$), with coefficient
$\prod_i b_{i f(i)}$ and exponent $\sum_i a_{f(i)}$. Reducing modulo $x^N - 1$
folds the exponents into $\Z_N$. The coefficient of $x^{\,p}$ therefore
collects exactly the $f$ whose multiset of exponents sums to $p$ modulo $N$ ---
and if $A$ is valid mod $N$, those $f$ are precisely the bijections. Since the
coefficients of a polynomial modulo $x^N - 1$ are recovered by the inverse
discrete Fourier transform of its values at the $N$-th roots of unity, this
writes scheme \textup{(F)} out explicitly:

\begin{proposition}[DFT permanent identity]\label{prop:dft}
Let $A$ be valid mod $N$ and $\omega = e^{2\pi i/N}$. Then for every
$B \in \C^{n \times n}$,
\begin{equation}\label{eq:dft}
\per B \;=\; \frac{1}{N} \sum_{r=0}^{N-1} \omega^{-pr}
\prod_{i=1}^{n} \Bigl( \sum_{j=0}^{n-1} b_{ij}\, \omega^{a_j r} \Bigr).
\end{equation}
\end{proposition}

\begin{proof}
By orthogonality of the characters of $\Z_N$,
$\frac1N \sum_{r=0}^{N-1} \omega^{(S - p)r} = [\,S \equiv p \pmod N\,]$.
Expanding the product inside \eqref{eq:dft} and exchanging sums,
\[
\text{RHS} \;=\; \sum_{f : \{1,\dots,n\} \to \{0,\dots,n-1\}}
\Bigl(\prod_{i} b_{i f(i)}\Bigr)
\cdot \Bigl[\, \textstyle\sum_i a_{f(i)} \equiv p \pmod N \,\Bigr].
\]
Writing $k_j = |f^{-1}(j)|$, the bracket condition is
$\sum_j k_j a_j \equiv p \pmod N$ with $\sum_j k_j = n$; validity forces
$k = (1,\dots,1)$, i.e.\ $f$ bijective. The surviving terms are exactly
$\sum_{\sigma \in S_n} \prod_i b_{i\sigma(i)} = \per B$.
\end{proof}

\begin{remark}[complex DFT or number-theoretic transform]\label{rem:ntt}
The proof uses only two properties of $\omega$: it has exact multiplicative
order $N$, and $N$ is invertible. Both hold in any field containing a
primitive $N$-th root of unity, not just $\C$ --- orthogonality is the
geometric sum $\sum_{r<N} \omega^{sr} = (\omega^{sN}-1)/(\omega^{s}-1) = 0$
for $s \not\equiv 0 \pmod N$. In particular, for a prime $q \equiv 1
\pmod N$ any element $\omega \in \F_q^{\times}$ of order $N$ makes
\eqref{eq:dft} hold verbatim in $\F_q$; the evaluation is then a size-$N$
\emph{number-theoretic transform} (NTT) and computes $\per B \bmod q$ for
integer $B$. Repeating over several such primes and reconstructing by the
Chinese remainder theorem recovers the exact integer permanent with no
rounding anywhere. We write $\omega = e^{2\pi i/N}$ here only
for concreteness.
\end{remark}

Validity is also \emph{necessary}, and checkably so:

\begin{proposition}[converse; the all-ones test]\label{prop:converse}
If $A$ is not valid mod $N$, then \eqref{eq:dft} fails for the all-ones matrix
$J_n$. More precisely, for $B = J_n$ the right-hand side of \eqref{eq:dft}
equals the number of functions $f$ with $\sum_i a_{f(i)} \equiv p \pmod N$,
namely
\[
\sum_{\substack{k \ge 0,\ \sum_j k_j = n \\ \sum_j k_j a_j \equiv p \ (N)}}
\binom{n}{k_0, \dots, k_{n-1}} \;\;\ge\;\; n! \;=\; \per J_n,
\]
with equality if and only if $A$ is valid mod $N$.
\end{proposition}

\begin{proof}
With $b_{ij} = 1$ every product $\prod_i b_{i f(i)}$ equals $1$, so as in the
proof of \cref{prop:dft} the right-hand side counts the functions $f$
satisfying the congruence; grouping them by multiplicity vector $k$ gives the
multinomial sum. The vector $k = (1,\dots,1)$ is always feasible and
contributes $n!$; any other feasible $k$ contributes its (positive) multinomial
coefficient, and one exists precisely when $A$ is invalid.
\end{proof}

Thus a single evaluation of \eqref{eq:dft} at $B = J_n$, compared against
$n!$, is a complete numerical validity test for $(A, N)$. Applied to the
super-increasing set it confirms validity at $N = \Nmin(n) = 2^n - 2^{\lfloor
\log_2 n \rfloor}$ numerically for all $n \le 32$ --- an independent check of
\cref{thm:minmod} at sizes far beyond the reach of enumerating the
$\binom{2n-1}{n}$ multisets.

\begin{remark}[translation invariance]\label{rem:translation}
Replacing every $a_j$ by $a_j + t \pmod N$ shifts each exponent sum
$\sum_i a_{f(i)}$ and the target $p$ by the same amount $tn$, so validity of
$(A, N)$ and the identity \eqref{eq:dft} are unaffected. Two normalizations of
the super-increasing set are therefore equally valid: offset $0$ (the set
$\{2^k - 1\}$ of \cref{thm:minmod}, containing $0$) and offset $1$ (the
geometric ladder $a_k = 2^k$), whose extra arithmetic structure a
computational implementation can exploit. Likewise $\omega$ may be replaced by
$\omega^{-1}$ (conjugating \eqref{eq:dft}, whose value is the real number
$\per B$ when $B$ is real).
\end{remark}

\paragraph{Cost.}
The identity \eqref{eq:dft} is a sum of $N \approx 2^n$ terms, each formed from
$n$ row sums against a shared table of powers of $\omega$; taken term by term
this costs $\Theta(2^n n^2)$ arithmetic operations. As with Ryser's and Glynn's
formulas, it drops to $\Theta(2^n n)$ --- the same asymptotic cost as those
formulas \cite{ryser1963,glynn2010,nijenhuis1978} --- only by traversing the
frequencies $r$ in an order in which consecutive terms differ incrementally, so
that the geometric-ladder structure of $a_k = 2^k$ lets each term reuse the
previous one's row sums rather than recompute them; this ordered reuse is the
exact analogue of the Gray-code traversal of the subsets in the classical
formulas, and no one of the three is thereby cheaper than the others. \Cref{thm:minmod}
is what makes the modulus, and hence the transform length, as small as
$N \approx 2^n$ in the first place.

\section{Conclusion}\label{sec:conclusion}

Ryser's formula, Glynn's formula, and the discrete Fourier transform are one
construction seen from three angles: orthogonal evaluation schemes that read the
permanent off as a single coefficient of the row-product polynomial
(\cref{prop:ortho}). Ryser and Glynn are the $0$--$1$ and $\pm 1$ faces of one
Hadamard system, orthogonal for every matrix; the scheme studied here
uses instead an order-$N$ \emph{cyclic} character, for which orthogonality is no
longer automatic but is the arithmetic condition that $(A, N)$ be valid. The
least $N$ that meets it --- the smallest cyclic transform representing the
permanent at all --- is the subject of the companion paper
\cite{fonollosa2026minmod}.

As algorithms the three are practically equivalent. Each sums on the order of
$2^n$ evaluations; each costs $\Theta(n^2)$ arithmetic operations per evaluation
when the points are taken in isolation ($n$ row-forms of $n$ terms), and each
lowers this to $\Theta(n)$ per point --- hence to $\Theta(2^n n)$ overall ---
only by traversing the points in an order in which consecutive evaluations differ
incrementally and reuse almost all of their work: the Gray-code order of the
subsets for Ryser and Glynn, and an analogous ordered traversal of the powers of
$\omega$ for the transform. None of the three is asymptotically or structurally
cheaper than the others. What the common framework contributes is therefore
conceptual, not computational. The orthogonality criterion (\cref{prop:ortho})
pins down \emph{exactly} which weighted evaluations compute the permanent, and
\cref{prop:abelian} identifies them: the character schemes of finite abelian
groups, and no more --- the criterion collapses $\prod_j v_{r,j}^{k_j}$ onto a
single dual character, which needs the multiplicative closure that group
characters supply, so an arbitrary orthogonal family does not qualify. Within
this class Ryser and Glynn are the two Boolean faces (\cref{prop:three}) and the
cyclic DFT is a third, new instance, with the transform the single degree of
freedom separating them. The algebraic price of admission is uniform: feasibility
and size are governed by \emph{unique multiset sums} in the chosen group, so exact
evaluation of the permanent becomes a combinatorial question about that group ---
the minimum valid modulus $\Nmin(n)$ for cyclic groups (\cref{thm:minmod}), and
the least order $|G|$ over all finite abelian groups, which is \emph{not} attained
cyclically: Glynn's own $2^{n-1}$-term form is the character scheme over
$(\Z_2)^{n-1}$, of order $2^{n-1}$, already below the least valid cyclic modulus
for $n \ge 3$ (\cref{rem:noncyclic}), and within a factor $O(\sqrt n\,)$ of the
partial-derivative lower bound $\binom{2n}{n}/2^n$ that every evaluation scheme
must obey (\cref{cor:nwbound}).

\bibliographystyle{plain}
\bibliography{refs}

\end{document}